\documentclass[a4paper,12pt]{article}

\usepackage{amsthm}
\usepackage{amsmath}
\usepackage{amssymb}

\numberwithin{equation}{section}

\newtheorem{thm}{Theorem}[section]

\newtheorem{lem}[thm]{Lemma}
\newtheorem{dfn}[thm]{Definition}

\def\R{\mathbb{R}}
\def\Ri{\mathbb{R}\cup\{+\infty\}}
\def\N{\mathbb{N}}
\def\eps{\varepsilon}

\def\epi{\mathrm{epi\,}}

\def\be{\begin{equation}}
\def\ee{\end{equation}}

\def\ba{\begin{array}}
\def\ea{\end{array}}

\DeclareMathOperator{\diam}{diam}

\begin{document}

\title{Simultaneous perturbed minimization of a convergent sequence of functions\thanks{The study  is supported by  the European Union-NextGenerationEU, through the National Recovery and Resilience Plan of the Republic of Bulgaria, SUMMIT project BG-RRP-2.004-0008-C01.}}
\author{Hristina Topalova\thanks{Faculty of Mathematics and Informatics, Sofia University,   5, James Bourchier Blvd, 1164 Sofia, Bulgaria, e-mail: htopalova@fmi.uni-sofia.bg}\ \ and Nadia Zlateva\thanks{Faculty of Mathematics and Informatics, Sofia University,   5, James Bourchier Blvd, 1164 Sofia, Bulgaria, e-mail: zlateva@fmi.uni-sofia.bg}}
\date{ }
\maketitle
\begin{abstract}
	We establish a general method for simultaneously perturbing a convergent sequence of functions in such a way that the sequence of strong minima of the perturbed functions tend to the strong minimum of
	their limit.
\end{abstract}

\section{Introduction}\label{sec:intro}
	One of the most challenging --- and thus rewarding in terms of new development --- topics in Optimization, Variational Analysis~\cite{rocwets,Tbook} and related areas, is the behaviour of the solutions, or approximate solutions, of a given type of problems under small perturbation, see e.g. \cite{Danzig}. In this, the value of the problem usually behaves way better, but what concerns the "argmin", all kind of "jumps" may occur.
	
	One approach of dealing with this issue is to perturb all the parameterized problems together in such a way that all of them become well-posed. Probably the first general result of this kind --- extending the theory of so called variational principles --- is due to P. G. Georgiev~\cite{pando}. This technique has been developed further since, e.g. \cite{dev-pro,vesel}.
	
	However, everywhere so far some convexity is assumed. In this work we focus on what would be the first step to general parametric perturbation method with a compact set of parameters. Namely, we take the parametric set to be $\{1,2,\ldots,\infty\}$ --- the one point compactification of the naturals --- which, of course, leads to a convergent sequence of functions. The principle merit of this work is that we impose no convexity assumptions.
	
	We will now present and explain our contribution. We work on a complete metric space $(X,d)$. Let us recall   that a function $f:X\to \Ri$  is called \emph{proper} when it is not identically equal to infinity and  \emph{lower semicontinuous} whenever its epigraph $\epi f:=\{ (x,r)\in X\times \R:f(x)\le r\}$ is a closed set. For  a proper   and bounded below function $f:X\to\Ri$ we set
	$$
		f_\varepsilon (x) := \inf\{f(y):\ d(y,x) \le \eps\}.
	$$	

	\begin{thm}\label{thm:minconv}
		Let $(X,d)$ be a complete metric space and let   $(\cal{P},\|\cdot\|_{\cal{P}})$ be a perturbation space on $X$.
		
		Let the functions  $f_n:X\to \Ri$, $n=1\dots\infty$, be proper, lower semicontinuous and bounded below, and such that
		\begin{equation}
			\label{eq:cond-1}
			\lim_{n\to\infty} f_n(x) = f_\infty(x),\quad\forall x\in X,
		\end{equation}
		\begin{equation}
			\label{eq:cond-2}
			\forall\varepsilon>0\ \exists N\in\mathbb{N}:\ f_n \ge (f_\infty)_\varepsilon - \varepsilon,\quad\forall n \ge N.
		\end{equation}
		Then there exists a dense $G_\delta$ subset $U$ of $\cal{P}$ such that for each $g\in U$ the functions $f_n+g$, where $n\ge 1$, attain strong minimum on $X$ at $x_n$ and
		\begin{equation}
			\label{eq:continuity-min}
			\lim_{n\to\infty} x_n = x_\infty,\quad\lim_{n\to\infty} f_n(x_n) = f_\infty(x_\infty).
		\end{equation}
  \end{thm}
  Perturbation spaces are defined in Section~\ref{sec:abstract}. Roughly speaking they allow perturbing a proper lower semicontinuous and bounded below function so that the perturbed function attains minimum. The idea is to do this simultaneously for all the functions  $(f_n)_{n\ge1}$ and then use the conditions \eqref{eq:cond-1} and \eqref{eq:cond-2} to derive the convergence of the corresponding minima.

  To give an idea how the above abstract theorem works, an immediate corollary to it is as follows. Let $(E,\|\cdot\|)$ be a Banach space with norm $\|\cdot\|$ which is Fr\'echet differentiable away from the origin. If $(f_n)_{n\ge1}$ satisfy \eqref{eq:cond-1} and \eqref{eq:cond-2} then for each $\eps > 0$ there is a Fr\'echet smooth function $g$ such that $|g(x)| < \eps$ and $\|g'(x)\| < \eps$ for all $x\in E$ and each $f_n+g$ has an unique minimum. Moreover, the sequence of minima converge as in \eqref{eq:continuity-min}.

  The reason behind the chosen convergence of a sequence of functions and especially \eqref{eq:cond-2} (as   \eqref{eq:cond-1} is just the pointwise convergence)  is a bit harder to explain. Note that uniform convergence satisfies easily \eqref{eq:cond-1} and \eqref{eq:cond-2}, so a reader who can do with uniform convergence is encouraged to think in those terms, which will simplify quite a few the considerations.

  In the context of Optimization --- thus Variational Analysis --- however, the uniform convergence can often be inadequate. The reason is that a sequence of indicator functions  of some sets $A_n$ (the indicator function to a set $A\subset X$ is equal to zero on $A$ and infinity outside) converges to the indicator function of $A$ if and only if $A_n$'s eventually coincide with $A$. As indicator functions are important in Optimization to accommodate the constraint sets, the standard uniform convergence will mostly not do. Note in this regard that, for example, the indicator functions of the intervals $[0,1/n]$ will indeed converge to the indicator function of $\{0\}$ in the sense of \eqref{eq:cond-1} and \eqref{eq:cond-2}.

  In the literature various geometrical (in the sense of using the epigraphs) notions have been developed to accommodate for the above situation. To keep it short we refer only to the seminal monograph of T. Rockafellar and  R. J.-B. Wets~\cite{rocwets}, but, of course, the latter sums up lots of previous development. For the current state of the art we address to the monograph of L. Thibault~\cite{Tbook}.  In this respect, it is easy to check that epigraph absorption (in other words, whatever metric is chosen in $X\times\mathbb{R}$, for each $\eps>0$ the epigraphs of $f_n$ are eventually contained in the $\eps$-enlargement of the epigraph of $f_\infty$; that is, each point of $\epi f_n$ is within $\eps$-distance to $\epi f_\infty$) imply the analytical condition \eqref{eq:cond-2}. Some uniformity, however, is required in that what the so called  {epi-convergence} will not do, see~\cite{rocwets}. Indeed, a simple example involving only piecewise linear convex functions on $\mathbb{R}$, illustrates how one may have $f_n$'s epi-convergent to $f_\infty$, but
 $\inf f_n \not\to  \inf f_\infty$, see \cite[p. 263]{rocwets}. The functions are $f_n(x) := \max\{ -1,x/n\}$ for $n\in\mathbb{N}$ and $x\in\R$, so $f_\infty \equiv 0$. It is easy to check that these cannot satisfy the conclusion of Theorem~\ref{thm:minconv}. Why we cannot just add the condition $\inf f_n \to  \inf f_\infty$ is because it is not preserved under addition of another (perturbation) function and our method crucially depends on something like Lemma~\ref{lem:can-add-uc} below.

The article is organized as follows. In the next section we study the properties of the convergence \eqref{eq:cond-1}, \eqref{eq:cond-2} that we need. In Section~\ref{sec:abstract} we present the version of the perturbation space and the corresponding perturbation method we use. After that completing the proof of Theorem~\ref{thm:minconv} in Section~\ref{sec:proof} takes a few lines.

\section{Lemmata}
	\label{sec:lem}	
	Here we study some properties of the convergence \eqref{eq:cond-1} and \eqref{eq:cond-2} that we will need.
	
	For a function $f:X\to \Ri$ and $\eps\ge 0$  we consider the sets of $\eps$-minima of $f$:
	$$
		\Omega _f (\eps)=\left\{ x\in X: f(x)\le \inf _X f+\eps\right\}.
	$$

	\begin{lem}
		\label{lem:lem-0}
		 Let $(X,d)$ be a metric space and let $f: X\to\Ri$ be proper and bounded below. Then for each $\delta > 0$
		\begin{equation}
			\label{eq:inf-eps-invar}
			\inf f_\delta = \inf f.
		\end{equation}
		Also, for each $0<\mu<\eps$
		\begin{equation}
			\label{eq:om-mu-eps}
			\forall x\in \Omega_{f_\delta}(\mu)\Rightarrow d(x,\Omega_f(\eps)) \le \delta.
		\end{equation}
	\end{lem}
	\begin{proof}
		Because $f\ge f_\delta$, it is clear that $\inf f \ge \inf f_\delta$. On the other hand, for any $x\in X$ by definition $f_\delta(x) = \inf\{f(y):\ d(y,x)\le\delta\} \ge \inf\{f(y):\ y\in X\} = \inf f$, so $\inf f_\delta \ge \inf f$, and \eqref{eq:inf-eps-invar} is verified.
		
		Let now $x\in\Omega_{f_\delta}(\mu)$, that is, $f_\delta (x) \le \inf f_\delta + \mu$. By \eqref{eq:inf-eps-invar}
		$$
			f_\delta (x) \le \inf f + \mu.
		$$
		By the definition of $f_\delta$, there is $y$ such that $d(y,x)\le\delta$ and $f(y) \le f_\delta(x) + \eps-\mu$. The above implies
		$$
			f(y) \le \inf f + \eps,
		$$
		that is, $y\in\Omega_f(\eps)$.
	\end{proof}
	\begin{lem}
		\label{lem:lim-inf}
		Let $(X,d)$ be a metric space. If the  proper and bounded below functions $(f_n)_{n\ge 1}$ satisfy \eqref{eq:cond-1} and \eqref{eq:cond-2} then
		\begin{equation}
			\label{eq:lim-inf}
			\lim_{n\to\infty} (\inf f_n) = \inf f_\infty.
		\end{equation}
	\end{lem}

	\begin{proof}
				From \eqref{eq:cond-2} and \eqref{eq:inf-eps-invar} it follows that for each $\varepsilon > 0$
		$$
			\inf f_n \ge \inf (f_\infty)_\varepsilon - \varepsilon = \inf f_\infty - \varepsilon,
		$$
		for all $n$ large enough. That is,
		\begin{equation}
			\label{eq:lim-inf-a}
			\liminf _{n\to\infty} (\inf f_n) \ge \inf f_\infty.
		\end{equation}
		On the other hand, fix arbitrary $\varepsilon > 0$ and let $\bar x\in X$ be such that
		$$
			f_\infty(\bar x) \le \inf f_\infty + \varepsilon.
		$$
		From \eqref{eq:cond-2} it follows that  $f_n(\bar x) < f_\infty(\bar x) + \varepsilon$ for all $n$ large enough, so
		$$
			\inf f_n \le f_n(\bar x) \le \inf f_\infty + 2\varepsilon
		$$
		for all $n$ large enough. Therefore,
		\begin{equation}
			\label{eq:lim-inf-b}
			\limsup _{n\to\infty} (\inf f_n) \le \inf f_\infty.
		\end{equation}
		Together \eqref{eq:lim-inf-a} and \eqref{eq:lim-inf-b} prove the claim \eqref{eq:lim-inf}.
	\end{proof}
	
	\begin{lem}\label{lem:4.1}
		Let $(X,d)$ be a metric space. If the  proper and bounded below functions $(f_n)_{n\ge 1}$ satisfy \eqref{eq:cond-1} and \eqref{eq:cond-2} then for each $\varepsilon > 0$
		\begin{equation}
			\label{eq:4.1}
			\forall x\in \Omega_{f_n}(\eps/4) \Rightarrow d(x, \Omega_{f_\infty}(\eps)) \le \eps/2,
		\end{equation}
		for all $n$ large enough. In particular,
		\begin{equation}
			\label{eq:4.1-limiting}
			\lim_{n\to\infty} \diam\bigcup_{k\ge n}\Omega_{f_k}(\eps/4) \le \diam \Omega_{f_\infty}(\eps) + \eps.
		\end{equation}
	\end{lem}

	\begin{proof}
		Note that \eqref{eq:4.1} is trivial for $n=\infty$, so we can focus on the natural $n$'s. Also, \eqref{eq:4.1-limiting} follows immediately from \eqref{eq:4.1}.
		
		Let $\eps>0$ be arbitrary and let $\delta = \eps/4$. Lemma~\ref{lem:lim-inf} is applicable here and  \eqref{eq:lim-inf} gives $N\in\mathbb{N}$ such that
		\begin{equation}
			\label{eq:4.1-delta-1}
			\inf f_n < \inf f_\infty + \delta,\quad\forall n\ge N.
		\end{equation}
		From \eqref{eq:cond-2} we can also assume that $N$ is so large that
		\begin{equation}
			\label{eq:4.1-delta-2}
			f_n \ge (f_\infty)_\delta - \delta,\quad\forall n\ge N.
		\end{equation}
		We claim that
		\begin{equation}
			\label{eq:4.1-claim}
			\Omega_{f_n}(\delta) \subset \Omega_{(f_\infty)_\delta}(3\delta),\quad\forall n\ge N.
		\end{equation}
		Indeed, fix a $n>N$ and a $x\in X$ such that $f_n(x) \le \inf f_n + \delta$. Then \eqref{eq:4.1-delta-2} gives
		$$
			(f_\infty)_\delta (x) \le f_n(x) + \delta \le \inf f_n + 2\delta,
		$$
		and \eqref{eq:4.1-delta-1} yields
		$$
			(f_\infty)_\delta (x) \le \inf f_\infty + 3\delta.
		$$
		But $\inf f_\infty = \inf (f_\infty)_\delta$  from \eqref{eq:inf-eps-invar} and so $x\in \Omega_{(f_\infty)_\delta}(3\delta)$, which proves \eqref{eq:4.1-claim}, because $x\in \Omega_{f_n}(\delta)$ was arbitrary.
		
		By applying \eqref{eq:om-mu-eps} to $f_\infty$ with $\mu=3\delta<\eps$, we can derive \eqref{eq:4.1} from \eqref{eq:4.1-claim}.
	\end{proof}
	
	We say that a proper and bounded below function $f:X\to\Ri$ attains its \emph{strong} minimum at $\bar x \in X$ if each minimising sequence converges to $\bar x$, that is,
	$$
		\lim_{n\to\infty} f(x_n) = \inf f\Rightarrow \lim_{n\to\infty} x_n = \bar x.
	$$
	It is clear that if $f$ attains a strong minimum then
	\begin{equation}
		\label{eq:str-min-crit}
		\lim_{\eps\to0}\diam\Omega_f(\eps) \to 0.
	\end{equation}
	
	\begin{lem}
		\label{lem:master-lemma}
		Let $(X,d)$ be a metric space. If the  proper and bounded below functions $(f_n)_{n\ge 1}$ satisfy \eqref{eq:cond-1} and \eqref{eq:cond-2}
		and each $f_n$ attains its strong minimum at $x_n$, then
		\begin{equation}
			\label{eq:min-converge}
			\lim_{n\to\infty} x_n = x_\infty,\quad \lim_{n\to\infty} f_n(x_n) = f_\infty(x_\infty).
		\end{equation}
	\end{lem}
	\begin{proof}
		The second part of \eqref{eq:min-converge} is just \eqref{eq:lim-inf}, so it follows from Lemma~\ref{lem:lim-inf}.
		
		For the first part, fix $\eps > 0$. From \eqref{eq:str-min-crit} for $f_\infty$ there is $\delta > 0$ such that
		$$
			\diam \Omega_{f_\infty}(\delta) + \delta < \eps.
		$$
		
		Now, this and \eqref{eq:4.1-limiting} of Lemma~\ref{lem:4.1}  imply
		$$
			\lim_{n\to\infty} \diam\bigcup_{k\ge n}\Omega_{f_k}(\delta/4) < \eps,
		$$
		but obviously $x_k\in\Omega_{f_k}(\delta/4)$, so
		$$
			d(x_k,x_\infty) \le \diam\bigcup_{k\ge n}\Omega_{f_k}(\delta/4),\quad \forall k \ge n.
		$$
		Therefore, there is $N\in\mathbb{N}$ such that
		$$
			d(x_k,x_\infty) < \eps,\quad \forall k \ge N.
		$$
The proof is then completed.
	\end{proof}

Let us recall that a function $g:X\to\mathbb{R}$ is \emph{uniformly continuous} if for any $\eps >0$ there exists $\delta >0$ such that $|g(x)-g(y)|<\eps$ whenever $d(y,x)<\delta$.
	
	\begin{lem}
		\label{lem:can-add-uc}
		Let $(X,d)$ be a metric space and let $g:X\to\mathbb{R}$ be uniformly continuous and bounded below.
		
		If the proper and bounded below functions $(f_n)_{n\ge 1}$ satisfy \eqref{eq:cond-2} then so do $(f_n+g)_{n\ge1}$.
	\end{lem}
	\begin{proof}
		Fix $\varepsilon > 0$. Since $g$ is uniformly continuous, there is $\delta\in(0,\varepsilon/2)$ such that
		$$
			|g(y)-g(x)| < \varepsilon/2,\quad\forall x,y\in X:\ d(y,x) < \delta.
		$$
		From \eqref{eq:cond-2} there is $N\in\mathbb{N}$ such that
		$$
			f_n \ge (f_\infty)_\delta - \delta,\quad\forall n\ge N.
		$$
		So, for each $x\in X$
		\begin{eqnarray*}
			(f_\infty+g)_\delta(x) &=& \inf\{f_\infty(y) + g(y):\ d(y,x)\le\delta\}\\
			&\le& \inf\{f_\infty(y) + g(x) + \varepsilon/2:\ d(y,x)\le\delta\}\\
			&=& (f_\infty)_\delta(x) + g(x) + \varepsilon/2.
		\end{eqnarray*}
		That is,
		$(f_\infty+g)_\delta \le (f_\infty)_\delta + g + \varepsilon/2$ and, therefore, for all $n\ge N$
		$$
			f_n+g \ge  (f_\infty)_\delta + g - \delta \ge (f_\infty+g)_\delta -\varepsilon/2 - \delta > (f_\infty+g)_\delta -\varepsilon.
		$$
		But $(f_\infty+g)_\delta\ge (f_\infty+g)_\varepsilon$, because $\delta < \varepsilon$, so
		$$
			f_n+g \ge  (f_\infty+g)_\varepsilon -\varepsilon,\quad\forall n\ge N,
		$$
and the claim follows.
	\end{proof}

 \section{Abstract perturbation method}\label{sec:abstract}

The authors of \cite{DGZ-article,DGZ} established not only a variant of Borwein-Priess Smooth Variational Principle, but also a general and flexible scheme for constructing similar variational principles or, what some prefer to call \emph{perturbation methods}. This scheme has been followed e.g. in \cite{dev-pro,iv-zla-jca,pertOrl} and we follow it here as well. Our approach ---  a bit more topological than that of \cite{DGZ-article,DGZ} --- with special attention to the level sets, can be traced back at least to \cite{cho-kend-rev}.

The basic problem is:  given a real Banach space $(E,\|\cdot\|)$ and a function $f:E\to \Ri$,  find a perturbation function $g:E\to \R $   such that the perturbed function  $f+g$ attains its minimum. Of course, the regularity of the perturbation --- continuous, Lipschitz, differentiable, etc. --- matters a lot for the applications.   The  idea to look for a perturbation in a specified Banach space of functions  comes from  the classical work of Deville, Godefroy and Zizler~\cite{DGZ-article,DGZ}. For example, \cite[Lemma 2.5]{DGZ} gives a list of conditions that when satisfied by a Banach space ${\cal P}$ of bounded   continuous functions on $E$ guarantees that  for each proper, lower semicontinuous and bounded below $f:E\to\Ri$, the set of all  $g\in{\cal P}$ such that the perturbed function $f+g$ attains its strong minimum on $E$ is a dense $G_\delta$ subset of ${\cal P}$.

More specifically,
if the  Banach space $E$  admits a Lipschitzian bump function which
is Fr\'echet differentiable (resp. G\^ateaux differentiable) away from the origin, then all functions $g$ on $E$ that are bounded, Lipschitzian and Fr\'echet differentiable
(resp. G\^ateaux differentiable), equipped with the norm $\| g\|  =\max\{ \|g\|_\infty, \|g'\|_\infty\}$
 constitute such a Banach space ${\cal P}_1$. Based on the latter one easily proves a variant of the  Smooth Variational Principle of Borwein and Preiss~\cite{B-P}. To prove the Ekeland Variational Principle~\cite{Ekeland} on arbitrary Banach space $E$, it is enough to consider the space ${\cal P}_2$ of all bounded Lipschitzian functions $g$ on $E$ equipped
with the norm
\[
\| g\| = \sup \{|g(x)|,x\in E\}+\sup \left\{ \frac{|g(x)-g(y)}{\|x-y\|}, x\neq y\right\}.
\]

In fact considering a perturbation space of functions  is a very flexible tool because the space of perturbations  can be chosen in a way which is most appropriate for the problem at hand. The question  for a given function $f:E\to \Ri$ to find a perturbation  $g:E\to \R $ belonging to appropriately selected Banach space of functions   such that the perturbation $f+g$ not only attains its minimum but preserves some distinctive properties of $f$, is explored in a series of papers: \cite{iv-zla-jca,iv-zla-jota,pertOrl}.

Here we will work with a very general definition of a perturbation space on a complete metric space. In a sense it is quite similar of that in  \cite{pertOrl}. Here we need additionally \emph{uniform continuity} to be able to apply Lemma~\ref{lem:can-add-uc}.

%[Definition 2.2]{pertOrl} but more convenient to our purposes.

\begin{dfn}\label{pert_sp}
Given a complete metric space $(X,d)$, a space
$(\cal{P},\|\cdot\|_{\cal{P}})$ of real uniformly continuous and bounded  on $X$ functions,  is called a \emph{perturbation space} on $X$ if
\begin{itemize}
\item[(i)]   $\cal{P}$ is complete with respect to the norm $\|\cdot\|_{\cal{P}}$ defined on $ \cal{P}$,  that dominates the uniform  convergence on $X$, that is, for some $c>0$
\begin{equation}
  \label{eq:dominate-uniform}
  \sup_{x\in X} |g(x)| \le c \|g\|_{\cal{P}},\quad\forall g\in  \cal{P}.
\end{equation}
In other words, $(\cal{P},\|\cdot\|_{\cal{P}})$ is a Banach space of bounded on $X$ uniformly continuous functions.
\item[(ii)]  For any $\eps >0$ there exists $\delta>0$    such that for any $x\in X$ there exists $g\in \cal{P}$ (depending on $x$) such that
\begin{equation}
  \label{eq:def:perturb:c}
  \|g\|_{\cal{P}}\le \varepsilon,\ x \in \Omega_g(\delta), \text{ and } \diam(\Omega_g(3\delta)) \le \varepsilon.
\end{equation}
\end{itemize}
\end{dfn}

It is routine to check  that the considered above  space  ${\cal P}_1$ (on appropriate Banach space $E$, of course) as well as the space ${\cal P}_2$ on arbitrary Banach space $E$ are perturbation spaces in the sense of the above definition.

We will need further a simple  result from \cite[Lemma 4]{iv-zla-jca}, see also \cite{pertOrl}: for any $f,g:X\to \Ri$, bounded below on $X$, and $\delta >0$,
\begin{equation}
  \label{eq:new-lema4}
  \Omega_f(\delta)\cap \Omega_g(\delta)\ne\varnothing\quad\Longrightarrow\quad
  \Omega_{f+g}(\delta)\subset\Omega_f(3\delta)\cap \Omega_g(3\delta).
\end{equation}

Following the lines of the proofs of \cite[Theorem 2.3]{pertOrl} and \cite[Theorem 2.5]{pertOrl} we get the following
\begin{thm}\label{thm:perturb-min}
  Let $(X,d)$ be a complete metric space and let $(\cal{P},\|\cdot\|_{\cal{P}})$ be a perturbation space on $X$. Let   $f:X\to \Ri$ be a  proper, lower semicontinuous and bounded below function.

  Then the set of all $g\in \cal{P}$ such that  $f+g$ attains its strong minimum  is a dense $G_\delta$ subset of ${\cal P}$. In particular for any $\eps>0$ there exists $g\in {\cal P}$ such that $\|g\|_{\cal P}\le \eps $ and $f+g$ attains its strong minimum.
\end{thm}

\begin{proof}
	We will use the basic fact that for a proper, lower semicontinuous and bounded below function $f$ on a complete metric space the condition \eqref{eq:str-min-crit} is also sufficient for $f$ to attain a  strong minimum.

 Consider for $n\in\N$ the subset $M_n$ of $\cal{P}$ defined by
\begin{equation}\label{eq:an-def}
M_n:=\left\{ g\in {\cal{P}}:  \exists\ t>0  :\ \diam\left(\Omega_{f+g}(t)\right) < \frac{1}{n}\right\}.
\end{equation}
We will show that $M_n$ is dense and open in $(\cal{P},\|\cdot\|_{\cal{P}})$. Then by Baire Category Theorem
$$
	U := \bigcap_{n=1}^\infty M_n
$$
will be dense in $X$ (it is $G_\delta$ by definition). For each $g\in U$ the definition of $M_n$ shows that \eqref{eq:str-min-crit} is satisfied for $f+g$ and, therefore, $f+g$ attains  strong minimum.

Fix $n\in\N$. Let $g\in M_n$ be arbitrary and let $\beta > 0$ be such that
$$
  \diam\left(\Omega_{f+g}(3\beta)\right) < \frac{1}{n}.
$$
For any $h\in \cal{P}$ such that $\|h\|_{\cal{P}}<\beta/2c$ it follows from \eqref{eq:dominate-uniform} that $h(x) - h(y) < \beta$ for any $x,y\in X$ and, therefore, $\Omega_h(\beta) = X$. From \eqref{eq:new-lema4} for $(f+g)$ and $h$ it follows that
$$
  \Omega_{f+g+h}(\beta) \subset \Omega_{f+g}(3\beta),
$$
thus $\diam (\Omega_{f+g+h}(\beta)) < 1/n$ and $g+h\in M_n$. So, $M_n$ is open.

Let now $h\in \cal{P}$ be arbitrary. Fix an arbitrary $\varepsilon \in (0,1/n)$.
Let $\delta > 0$ be provided by (ii) of Definition~\ref{pert_sp}. Fix
$$
  x \in \Omega_{f+h} (\delta),
$$
and let $g\in \cal{P}$ satisfy \eqref{eq:def:perturb:c}. The latter implies that $x\in \Omega_g(\delta)$. From  \eqref{eq:new-lema4} it follows that
$$
  \Omega_{f+h+g}(\delta) \subset \Omega_g (3\delta) \Longrightarrow \diam (\Omega_{f+h+g}(\delta)) \le \varepsilon.
$$
As $\varepsilon<1/n$, this means that $h+g\in M_n$. Since $\|g\|_{\cal{P}}\le\varepsilon$, see \eqref{eq:def:perturb:c}, the distance from $h$ to $M_n$ is smaller than $\eps$. In other words, $M_n$ is dense in~$\cal{P}$.
\end{proof}

\section{Proof of Theorem~\ref{thm:minconv}}\label{sec:proof}

	From Theorem~\ref{thm:perturb-min} for each $n\ge 1$ there is a dense $G_\delta$ subset $U_n$ of $\cal{P}$ such that for each $g\in U_n$ the function $f_n+g$ attains strong minimum. 	
	Let
	$$
		U := \bigcap_{n\ge1} U_n.
	$$
	By Baire Category Theorem, $U$ is a dense $G_\delta$ subset of $\cal{P}$. Let $g\in U$. This means that all $(f_n+g)_{n\ge 1}$ attain strong minimum. Let $x_n$ be the strong minimum of $f_n+g$. Since $g$ is uniformly continuous, by Lemma~\ref{lem:can-add-uc} the functions $(f_n+g)_{n\ge1}$ satisfy \eqref{eq:cond-2}. Clearly, they satisfy \eqref{eq:cond-1} as well and so, Lemma~\ref{lem:master-lemma} for $(f_n+g)_{n\ge1}$ yields
	$$
		\lim_{n\to\infty} x_n = x_\infty,\quad \lim_{n\to\infty} f_n(x_n) + g(x_n) = f_\infty(x_\infty) + g(x_\infty).
	$$
	But then $g(x_n)\to g(x_\infty)$ and \eqref{eq:continuity-min} is verified.

\bigskip

\textbf{Acknowledgements.}
We wish to express our deep respect  to Dr. Milen Ivanov who pointed  us to this area of research and to give him our thanks  for his guidance  and for the invaluable help and support.

\end{document}